\documentclass[10pt,journal]{IEEEtran}

\usepackage{amsmath,amssymb,amsbsy,graphics,float}
\usepackage{graphicx,xcolor,pstricks}
\usepackage{nomencl}
\makenomenclature

\IEEEoverridecommandlockouts                              
  \def\rep#1{(\ref{#1})}                                                        
\usepackage{multirow}
\newtheorem{defini}{Definition}
\newtheorem{thm}{Theorem}
\newtheorem{remark}{Remark}

\newtheorem{pro}{Proposition}

\title{\LARGE \bf
Perturbation Analysis of the Wholesale Energy Market Equilibrium in the Presence of Renewables}

\author{Arman Kiani and Anuradha Annaswamy
\thanks{This work was supported by the Technische Universit\"at M\"unchen - Institute for
Advanced Study, funded by the German Excellence Initiative and by
Deutsche Forschungsgemeinschaft (DFG) through the TUM International
Graduate School of Science and Engineering (IGSSE).}
\thanks{Arman Kiani is with the Institute of Automatic Control
  Engineering, Technische Universit\"at M\"unchen, D-80290 Munich,
  Germany. {\tt\small arman.kiani@tum.de}}%
\thanks{Anuradha Annaswamy is with the Department of Mechanical
  Engineering, Massachussets Institute of Technology, Cambridge, MA
  02319, USA. {\tt\small aanna@mit.edu}}
}

\begin{document}


\maketitle
\thispagestyle{empty}
\pagestyle{empty} 

\begin{abstract}
 One of the main challenges in the emerging smart grid is the
 integration of renewable energy resources (RER).
The latter introduces both intermittency and uncertainty into the grid, both of which can affect the underlying energy market. An interesting concept that is being explored for mitigating the integration cost of RERs is Demand Response. Implemented as a time-varying electricity price in real-time, Demand Response has a direct impact on the underlying energy market as well. 
Beginning with an overall model of the major market participants together with the constraints of transmission and generation, we analyze the energy market in this paper and derive conditions for global maximum using standard KKT criteria.  The effect of uncertainties in the RER on the market equilibrium is then quantified, with and without real-time pricing. Perturbation analysis methods are used to compare the equilibria in the nominal and perturbed markets. These markets are also analyzed using a game-theoretic point of view. Sufficient conditions are derived for the existence of a unique Pure Nash Equilibrium in the nominal market. The perturbed market is analyzed using the concept of closeness of two strategic games and the equilibria of close games. This analysis is used to quantify the effect of uncertainty of RERs and its possible mitigation using Demand Response. Finally numerical studies are reported using an IEEE 30-bus to validate the theoretical results.
%

%

\end{abstract}


\makenomenclature
\nomenclature{$P_{Gib}$}{Power block $b$ that the generating unit $i$ is producing}
\nomenclature{$P_{Djk}$}{Power block $k$ that the demand $j$ is consuming}
\nomenclature{$\delta_n$}{voltage angle of bus $n$}
\nomenclature{$ \rho_n$}{Locational marginal price corresponding to the generating unit $i$ or the demand $j$ that is located at node $n$ }
\nomenclature{$\alpha_i$}{Dual variable associated with the maximum capacity constraint of generating unit $i$}
\nomenclature{$\phi_{ib}$}{Dual variable associated with the maximum capacity limit for block $b$ of generating unit $i$}
\nomenclature{$\sigma_j$}{Dual variable associated with the minimum demand constraint of demand $j$}
\nomenclature{$ \psi_{jk}$}{Dual variable associated with the maximum capacity limit for block $k$ of demand $j$}
\nomenclature{$ \gamma_{nm}$}{Dual variable associated with the transmission capacity constraint of line $n- m$}

%

\nomenclature{$B_{nm}$}{susceptance of line $n-m$}
\nomenclature{$P_{Gi}^{max}$}{Maximum power output of generating unit $i$}
\nomenclature{$P_{Gi}^{min}$}{Minimum power output of generating unit $i$}
\nomenclature{$P_{Gib}^{max}$}{Maximum power output in block $b$ of generating unit $i$}

\nomenclature{$P_{Djk}^{min}$}{Minimum power supplied to the demand $j$}
\nomenclature{$P_{Djk}^{max}$}{Maximum power demanded in block $k$ of demand $j$}
\nomenclature{$P_{nm}^{max}$}{Transmission capacity limit of line $n-m$}


\nomenclature{$\lambda_{Djk}^B$}{Price bid by demand $j$ to buy power block $k$}
\nomenclature{$\lambda_{G_{ib}}^B$}{Price bid by generating unit $i$ to sell power block $b$}
\nomenclature{$\lambda_{D_{jk}}^U$}{Marginal utility associated with block $k$ of demand $j$}
\nomenclature{$\lambda_{G_{ib}}^C$}{Marginal operating cost associated with block $b$ of generator $i$}

\nomenclature{$N_D$}{Number of demands}
\nomenclature{$N_{Dj}$}{Number of blocks demanded by demand $j$}
\nomenclature{$N_{G}$}{ Number of generating units}
\nomenclature{$N_{Gi}$}{Number of blocks bid by generating unit $i$}
\nomenclature{$N_T$}{Number of time periods}

\nomenclature{$D$}{Number of demands}
\nomenclature{${\cal G}$}{Set of indices of generating units $\{1,2, \ldots , N_G\}$}
\nomenclature{${\cal T}$}{ Set of considered time periods $\{T_1, T_2,  \ldots , T_{N_T}\}$}
\nomenclature{$\Omega_n$}{Set of indices of nodes connected to node $n$}
\nomenclature{ $\theta$ }{Set of indices of generating units at node $n$}
\nomenclature{ $\vartheta$ }{Set of indices of demands at node $n$}
%

\printnomenclature
\section{Introduction}
The goal in a Smart Grid is to integrate large-scale renewable generation and emerging storage technologies, control signals to loads to match supply, and dramatically improve energy efficiency. Electricity markets, the entity that carries out power balance by scheduling power using bids from various generating companies, are crucial components that can contribute to an efficient smart grid. Electricity market models that accurately represent the market dynamics are exceedingly important. These models must capture the behavior of the dominant market players such as generators, consumers, and ISO, market mechanisms such as gaming behaviors (ex. bidding strategies), real-time prices via Demand Response, diverse dynamic drivers (e.g., weather, load, fuel prices, and wind supply), and physical constraints (e.g., ramping, transmission congestion) \cite{kiani2011}. These market models and analysis of the market equilibrium directly help to identify various sources of price volatility and dependence and sensitivity to uncertainties and intermittencies in renewable energy sources and variations in real-time prices due to Demand Response. Our focus in this paper is on an electricity market model that captures the dynamics of market players, physical constraints, perturbations in renewable energy resources (RER), and variations due to demand response. 

Various methods have been proposed in the literature to determine market models \cite{Oren2008, Oren2007, conejo2009}. Given the primary purpose of market balance, these methods are primarily focused on the analysis of market equilibrium. Methods in \cite{Oren2008, Oren2007} model the electricity market participants subject to spot market equilibrium and their own constraints, leading to a method termed Equilibrium Programming with Equilibrium Constraints (EPECs). In such models, each agent solves a Mathematical Program with Equilibrium Constraints (MPEC) \cite{conejo2009}, and by using stationarity theory for MPECs a standard mixed complementarity problem (CP) is shown to lead to equilibrium \cite{hu2007}.

Another approach to market modeling employs a variational and complementarity Problem (CP) formulation (see \cite{Hobbs2001, Hobbs2007, Hobbs2000, Metzler2003}) which models the equilibrium of the electricity market where forward and spot decisions are made simultaneously. These models may also be specified by the ISOÕs objective,  maximization of social welfare \cite{Oren2008, Oren2007} or maximization of wheeling or transmission revenue \cite{Hobbs2001, Metzler2003}.  A single-settlement framework based on a linear complementarity problem (LCP) is proposed in  \cite{conejo2006}, which leads to a strategic game in which players have coupled constraints and the equilibrium properties are analyzed.
Yet another approach used for electricity markets is the use of Cournot models.  Here, the strategies of the generator companies are allowed to depend on other market participants, who also act so as to maximize their own profits. This in turn allows a more transparent dependence of costs of generators and consumersÕ profit on the market equilibrium (see \cite{Berry99, Cardell97, Hobbs2000, Nasser98, Wei99, Cunningham02, kiani2011, kianicdc2010, Alvarado2000, Tang2001}).

In the context of a market for smart grid, the question that needs to be addressed is how the effect of RERs can be best captured in a market model. The integration of RER poses the challenge of intermittency and uncertainty, both of which in turn can affect the economic planning and operation of the overall power system. The uncertainties can impose costs due to the limited-dispatchability of intermittent generation \cite{Fabbri2005}, the variability in resource availability, and errors in forecasting resource availability and loads \cite{DeMeo2005, smith2007,DeMeo2007}. There is a considerable literature covering many important aspects of wind power ranging from comprehensive integration studies, forecasting methods, and technology issues (see \cite{hamid2010, Bott2009, Bouffard2008, DeMeo2005, smith2007}). In the specific context of the incorporation of renewable energy in a market, the results of \cite{ Sioshansi2009, Bouffard2005, Bouffard22005, Bitar11, Hetzer2008, conejowind2009} are more pertinent, all of which deal with a perturbed market due to RER. In these papers, a stochastic framework is used to capture both overestimation and underestimation of available wind power on the optimal expected profit of wind power producer, the overall market efficiency, and the overall operation cost. In this paper, we utilize a deterministic version of the the model developed in \cite{Hetzer2008} and analyze the perturbed market and quantify the effect of wind uncertainty on the market equilibrium.

In recent years, there is a significant interest in Demand Response, a concept that allows the electricity demand to response to fluctuations in the electricity price \cite{Sioshansi2009}, \cite{hogan2009}. One way to reduce wind integration costs is to introduce demand response in the form of time-variant retail electricity rate, such as real-time pricing (RTP). RTP can potentially reduce wind integration and forecast error costs, since consumer demand could be made to follow the supply of wind generation by using a price signal. Under RTP, if available wind generation is less than forecast, the high cost of deploying ancillary services to cover the generation shortfall can reduce electricity demand and the cost of serving the load \cite{Sioshansi2010}. Similarly, because wind generation has zero marginal cost, electricity demand may increase when there is more wind resource available than forecast.

In this paper, we discuss the equilibrium of an electric energy market in the presence of renewable energy sources and demand response. Starting from the framework proposed in \cite{conejo2006}, we use the approach based on linear complementarity problem to delineate the underlying market model. Representing the effect of forecast errors in renewable energy as an uncertain parameter, we evaluate its effect on the market equilibrium. In a similar manner, focusing on RTP as the means by which demand response is realized, and representing the effect of RTP as a parametric curtailment on the demand, the effect of this parameter on the overall market equilibrium is also studied. A comprehensive market framework is introduced in order to evaluate these uncertainty effects. This market framework consists of three main participants, generating companies including renewable energy sources, consumer companies that are capable of responding to market changes, and independent system operators (ISO) who clear the market by maximizing social welfare \cite{conejo2006} and \cite{chao1998}.  The overall market equilibrium is first evaluated in the absence of uncertainties due to either forecast errors in renewable energy sources or due to demand response, and sufficient conditions for the proposed market equilibrium to be identical to the Nash equilibrium is established.  The market response in the presence of uncertainties, and in particular, the market equilibrium is analyzed using perturbation analysis. A parametric characterization of the equilibrium shift is delineated. In both the nominal and perturbed markets, the equilibriums are derived using KKT conditions We next define formally the closeness of two strategic game and relate the equilibrium of close games using the notion of $\alpha-\text{approximation}$ and $\epsilon-\text{equilibrium}$. IEEE 30-bus system is used to evaluate the results of both uncertainties in renewable energy and in demand curtailment.

The remainder of this paper is organized as follows: In Section \ref{pre} we present the necessary preliminaries, Section \ref{model} describes the model of the three market participants. In Section \ref{equi}, the overall market equilibrium under nominal conditions is formulated. A perturbation analysis is introduced in Section \ref{pert} to address the effects due to uncertainty, and the resulting effects are summarized.  In Section \ref{sim} numerical simulation results are reported. Section \ref{conc} includes concluding remarks.

\section{Preliminaries}
\label{pre}
In this section, we provide some preliminaries related to the wholesale energy market structure. These include the fundamental theorems of convex optimization which are presented in Sections  \ref{DD}. Given the close relationship between convex optimization 
and Game theory, we present in Section \ref{nash} related definitions and theorems related to game theory and Nash equilibrium.
In particular, the link between the convex optimization problem and uniqueness of the Pure Strategy Nash Equilibrium is presented including sufficient conditions for the latter.  These in turn are directly used in establishing the equilibrium of the wholesale market in Section \ref{model}.
We start with a few basic definitions.
\begin{defini}
A set $K \subseteq \mathbb{R}$ is convex if for any two points $x, y \in K$, 
\begin{equation}
\alpha x + (1- \alpha) y \in K, \; \forall x, y \in K \; \text{and} \; \alpha \in [0,1].
\end{equation}
\label{convexset}
\end{defini}
\begin{defini}
Given a convex set $K \subseteq \mathbb{R}$ and a function $f(x): K \rightarrow \mathbb{R}$; $f$ is said to be a convex function on $K$ if, $\forall x,y \in K$ and $\alpha \in (0,1)$,
\begin{equation}
f(\alpha x + (1-\alpha)y) \leq \alpha f(x) + (1-\alpha)f(y),
\end{equation}
Furthermore, a function $f(x)$ is concave over a convex set if and only if the function $ -f(x)$ is a convex function over the set.
\end{defini}
\begin{defini}
\label{def3}
Given a scalar-valued function $f(x): \mathbb{R}^n \rightarrow \mathbb{R}$ we use the notation $\nabla f(x)$ to denote the gradient vector of $f(x)$ at point $x$, i.e.,
\begin{equation}
\nabla f(x) =  \begin{bmatrix} \frac{\partial f(x)}{\partial x_1}, ... , \frac{\partial f(x)}{\partial x_n} \end{bmatrix}^T.
\label{gv}
\end{equation}
\end{defini}
\begin{defini}
Given a scalar-valued function $f(x): \prod_{i=1}^I \mathbb{R}^{m_i} \rightarrow \mathbb{R}$ we use the notation $\nabla_i f(x)$ to denote the gradient vector of $f(x)$ with respect to $x_i$ at point $x$, i.e.,
\begin{equation}
\nabla_i f(x) =  \begin{bmatrix} \frac{\partial f(x)}{\partial x_i^1}, ... , \frac{\partial f(x)}{\partial x_i^{m_i}} \end{bmatrix}^T.
\end{equation}
\end{defini}
\subsection{Convex Optimization}
\label{DD}
Consider a generic optimization problem  
\begin{equation}
\begin{aligned}
& \text{Maximize} \quad f(x) \\
& \text{s.t.} \;
g_n(x) = 0, \quad \forall n=1, \ldots, N \\
&\sum_{n=1}^N R_{mn}h_n(x) \geq c_m, \quad \forall m=1, \ldots L
\label{moptpro}
\end{aligned}
\end{equation}
where $f(x)$ is called the objective function or cost function, $R$ is a matrix of constants and $c_m$ are constants. 
We assume that $f(x): \mathbb{R}^n \rightarrow \mathbb{R}$ 
is a convex function to be maximized over the variable $x$, the functions $g_n(x)$ as equality constraints are affine, and the functions $h_n(x)$ as inequality constraints are cocave. With these assumptions
the optimization problem \rep{moptpro} is termed a convex
optimization problem.
 In addition, the constraint set for
the optimization problem is convex which allows us to use the
method of Lagrange multipliers and the Karush Kuhn Tucker (KKT)
theorem which we state below \cite{Bertsekas2001, Bertsekas1999}.
\begin{thm}
\label{kkt}
Consider the optimization formulated in \rep{moptpro}, where $f(x)$ is a convex function, $g_n(x)$ are affine functions, and $h_n(x)$ are cocave functions. Let $x^*$ be a feasible point, i.e. a point that satisfies all the constraints. Suppose there exists constants $\lambda_n$ and $\mu_m \geq 0$ such that 
\begin{equation}
\begin{aligned}
& \nabla f(x^*)  + \sum_{n=1}^N \lambda_n \nabla g_n(x^*) + \\
& \sum_{m=1}^L \mu_m (R_{mn} \nabla h_n(x^*) - c_m) = 0 \quad  \forall n = 1 ... N \\
& \mu_m( R_{mn}h_n(x^*) - c_m) = 0  \quad, \forall m=1, \ldots, L 
\label{lagrangcompl}
\end{aligned}
\end{equation}
then $x^*$ is a global maximum. If $f(x)$ is strictly concave then $x^*$ is also the unique global maximum.
\end{thm} 
\begin{proof}
see \cite{Bertsekas1999}.
\end{proof}
\subsection{Game Theory and NASH Equilibrium}
\label{nash}
An alternative game-theoretic approach can be used to describe the market equilibrium \cite{Torre2003, Hobbs2000}. To describe a game, there are four things
to consider: 1) the players, 2) the rules of the game, 3) the outcomes
and 4) the payoffs and the preferences (utility functions)
of the players. 
A player plays a game through actions. An action is a choice or
election that a player takes, according to his (or her) own strategy.
Since a game sets a framework of strategic interdependence, a participant
should be able to have enough information about its own
and other players' past actions. This is called the information set.
A strategy is a rule that tells the player which action(s) it
should take, according to its own information set at any particular
stage of a game. Finally, a payoff function expresses the utility that
a player obtains given a strategy profile for all players.
More formally a strategic form game is defined as follows.
\begin{defini}
A strategic forms game is a triplet $\langle \mathcal{I}, (S_i)_{i \in \mathcal{I}} , (u_i)_{i \in \mathcal{I}} \rangle$ such that $\mathcal{I}$ is a finite set of players , i.e. $\mathcal{I} = \{ 1,...,I \}$, $S_i$ is the set of available actions for player $i$, $s_i \in S_i$ is an action for player $i$, and finally $u_i: S \rightarrow R$ is the payoff function of player $i$ where $S = \prod_i S_i $ is the set of all action profiles.
\end{defini}
In addition, we use the notation $s_{-i} = [s_j]_{j \neq i}$ as a vector of actions for all players except $i$, $S_{-i}= \prod_{j \neq i} Sj$ as the set of all action profiles for all players except $i$, and finally $(s_i, s_{-i})\in S$ denoted a strategy profile of the game. Informally, a set of strategies is a Nash equilibrium if no player can not do better by unilaterally changing his or her strategy.
\begin{defini}
\label{ne}
A Pure Strategy Nash Equilibrium of a strategic game $\langle \mathcal{I}, (S_i)_{i \in \mathcal{I}} , (u_i)_{i \in \mathcal{I}} \rangle$ is a strategy profile $s^* \in S$ such that for all $i \in  \mathcal{I}$
\begin{equation} 
u_i(s_i,s^*_{-i}) \leq u_i(s^*_i,s^*_{-i})\; \forall s_{i} \in S_{i}.
\label{NE}
\end{equation}
\end{defini}
\begin{defini}
Let us assume that for player $i \in \mathcal{I}$, the strategy set $S_i$ is given by
\begin{equation}
S_i = \{x_i \in \mathbb{R}^m | h_i(x_i) \geq 0 \}
\label{strategy}
\end{equation}
where $h_i: \mathbb{R}^m_i \rightarrow \mathbb{R}$ is a concave function and the set of strategy profiles $S= \prod_{i=1}^I S_i \subset \prod_{i=1}^I \mathbb{R}^m_i$ is a convex set. The payoff functions $(u_1, ... , u_I)$ are said to be diagonally strictly concave for $x \in S$ if for every $x^*, \hat{x} \in S$, we have
\begin{equation}
(\hat{x} - x^*)^T \nabla u(x^*) + (x^* - \hat{x} )^T \nabla u(\hat{x}) > 0.
\end{equation}
\label{concave}
\end{defini}

In the current context, the optimization problem in \rep{moptpro} can be
viewed as a game played by all player $i\in \mathcal{I}$, and so that
together they can arrive at a Nash equilibrium that satisfies the
condition \rep{lagrangcompl} in theorem \ref{kkt}.

We now present a theorem that links the equilibrium point of the convex optimization problem in \rep{moptpro} to Nash equilibrium using the above definitions.
\begin{thm}
\label{thm2}
Consider a strategic form game $\langle \mathcal{I}, (S_i)_{i \in \mathcal{I}} , (u_i)_{i \in \mathcal{I}} \rangle$  with $S_i$ as in Definition \ref{concave}. 
Assume that the symmetric matrix $(U(x) + U^T(x))$ is negative definite for all $x \in S$,
where the matrix $U(x)$ is the Jacobian of the gradient vector of $u(x)$, i.e. the $j$th column of $U(x)$ is $\partial u(x) / \partial x_j$, $j = 1,..., I$. Then the game has a unique pure strategy Nash equilibrium identical to the optimal solution $x^* \in \prod_{i=1}^I \mathbb{R}^m_i$ of the following optimization problem
\begin{align}
& \nonumber \max_{y_i \in \mathbb{R}^m_i} u_i (y_i, x^*_{-i})\\
& \text{s.t.} \; h_i(x_i) \geq 0.
\label{thmopt}
\end{align}
\end{thm}
\begin{proof}
From Theorem 6 in \cite{rosen65} we have that if $U(x)+U^T(x)$ is negative definite for all $x\in S$, then the payoff functions $(u_1,...,u_I)$ are diagonally strictly concave.
This in turn, using Theorem 2 in \cite{rosen65}, implies that there is a unique pure strategy Nash equilibrium that is identical to the optimal solution $x^* \in \prod_{i=1}^I \mathbb{R}^m_i$ of the convex optimization problem in \rep{thmopt}.
\end{proof}

\section{Modeling of Market Agent Behavior}
\label{model}
 In this section, we derive the model of the overall electricity market. The main components of this market include
(i) Generating Companies (GenCo)
(ii) Consumer Companies (ConCo)
(iii) Independent System Operator (ISO)  \cite{conejo2006}.
We focus our attention in this paper on the wholesale market, and use a deterministic framework to represent all components.
The underlying problem that is to be addressed is the determination of the power production of the generating 
units and the power demanded by the consumers such that power is balanced at all network nodes and social welfare 
is maximized while meeting capacity limits. This is a fairly complex optimization problem since each of the components (i)-(iii) are subjected to constraints that may conflict with each other. 
For instance, consumersÕ are interested in maximizing their benefit, and  are therefore interested in cheap electricity. Availability of such cheap generation is limited due to capacity constraints on transmission lines.  In this case, the ISO may need to force more expensive units to keep things in equilibrium.
The solution of the resulting optimization problem not only determines the optimal dispatch but also the 
optimal Local Marginal Price which is determined as the corresponding shadow price.
 In what follows, we model each of the components (i)-(iii) together with their constraints and the optimization goals.
\subsection{Generating Company}
It is assumed that the generating company consists of $N_G$ generating units, and that the production of each generating unit $i$, $i=1,\ldots, N_G$ is divided into $b$
power blocks, where $b=1,\ldots,N_{G_i}$, and is denoted as  $P_{Gib}$. The associated linear operating cost
is denoted as $\lambda_{G_{ib}}^C$.
The goal of the company is to maximize its overall profit, and is stated as 
\begin{equation}
maximize \;  \sum_{i \in G_f}\sum_{b=1}^{N_{Gi}} \big( \rho_{n(i)}-\lambda_{G_{ib}}^C \big) P_{Gib}
\label{gencooverall}
\end{equation}
where 
$\rho_{n(i)}$ denotes the LMP of unit $i$ at node $n$ in the network. The power production is subject to the following constraints:
\begin{subequations}
\begin{align}
& \sum_{b=1}^{N_{Gi}}P_{Gib} \leq P_{Gi}^{max}:  \alpha_i;  \forall i\in G_f \label{maxgen} \\
& P_{Gib} \leq P_{Gib}^{max}:  \phi_{ib};  \forall i\in G_f;b=1, \ldots,N_{Gi} \label{maxgeneach} \\
& P_{Gib} \geq 0 ;  \forall i\in G_f;b=1, \ldots,N_{Gi} \label{posgeneach}
\end{align}
\end{subequations}
where $\alpha_i$ and $\phi_{ib}$ are the corresponding shadow prices.
The decision variables of this problem are the amounts of power $P_{Gib}$ to be
generated by each unit $i$ in each block $b$. While $\rho_{n(i)}$ are variables in the overall OPF problem, 
they can be viewed as fixed values with respect to the optimization of \rep{gencooverall}.
Necessary and sufficient conditions for the optimization of \rep{gencooverall} subject to \rep{maxgen} - \rep{posgeneach} are enumerated below, using dual variables $\alpha_i$ and $\phi_{ib}$, which correspond to the KKT conditions \cite{conejo2006}
\begin{subequations}
\begin{align}
&P_{Gib} \big( -\rho_{n(i)} + \lambda_{G_{ib}}^C  + \alpha_i + \phi_{ib} \big) = 0 \; \forall i\in G_f; \label{Gopt1eq} \\
&\nonumber  b=1, \ldots,N_{Gi}  \\
& \alpha_i \big( -\sum_{b=1}^{N_{Gi}}P_{Gib} + P_{Gi}^{max}\big) = 0 ; \forall i\in G_f \label{Gopt2eq} \\
& \phi_{ib} \big(P_{Gib}^{max} - P_{Gib}\big) = 0 ; \forall i\in G_f;b=1, \ldots ,N_{Gi} \label{Gopt3eq}
\end{align}
\end{subequations}
\begin{defini}
\label{RGECO}
GenCo $i$ is said to behave rationally if the associated linear operating cost denoted as $\lambda_{G_{ib}}^C$ is nondecreasing with respect to $P_{G_{ib}}$. That is
\begin{equation}
d(\lambda_{G_{ib}}^C)/d (P_{G_{ib}}) > 0, \; \forall i\in G_f, \; b=1, \ldots,N_{Gi}.
\label{cond1}
\end{equation}
\end{defini}
\subsection{Consumer Modeling}
A consumer company (ConCo) is assumed to consist of $N_D$ units, and the demand of each unit $j$, $j=1,\ldots, N_D$, is divided into several blocks $N_{D_j}$, with a block-index $k=1,\ldots,N_{D_j}$. The associated linear utility function is denoted as $\lambda_{D_{jk}}^U$ which represents the value of using electricity for the consumer. The goal of a ConCo is to maximize the total profit while consuming electricity. 
This profit, for a unit $j$ in block $k$ connected to node $n$, is determined as the difference between the utility $\lambda_{D_{jk}}^U$ and the corresponding LMP $\rho_n(j)$ \cite{conejo2006}. Assuming that the corresponding power consumed is denoted as $P_{Djk}$, the maximization problem can be posed as 
\begin{equation}
Maximize \; \sum_{j \in D_q}\sum_{k=1}^{N_{Dj}} \big( \lambda_{D_{jk}}^U-\rho_{n(j)}\big)P_{Djk}
\label{objd}
\end{equation}
As before, this is subjected to the constraints
\begin{subequations}
\begin{align}
& P_{Djk} \leq P_{Djk}^{max}: \sigma_{j} ; \forall j\in D_q;k=1, \ldots,N_{Dj} \label{maxd} \\
& \sum_{k=1}^{N_{Dj}}P_{Djk} \geq P_{Djk}^{min}:  \psi_{jk} ; \forall j\in D_q;k=1, \ldots,N_{Dj} \label{mind} \\
& P_{Djk} \geq 0 ; \forall j\in D_q;k=1, \ldots,N_{Dj} \label{nonzerod} 
\end{align}
\end{subequations}
where $\sigma_j$ and $\psi_{jk}$ are the corresponding dual variables for \rep{maxd} and \rep{mind}. While the constraint \rep{nonzerod} is almost always satisfied in current-day markets, this may not be the case in micro-grid topologies.
The decision variables of this problem are  $P_{Djk}$, the amounts of power to be
consumed by each demand $j$ in each block $k$.
While $\rho_{n(j)}$ are  variables in the overall OPF problem, 
they can be viewed as fixed values, as before, with respect to the optimization of \rep{objd}.
Necessary and sufficient conditions for the optimization of \rep{objd} subject to \rep{maxd} and \rep{mind} are enumerated below, using dual variables $\sigma_j$ and $\psi_{jk}$, which correspond to the KKT conditions \cite{conejo2006}
\begin{subequations}
\begin{align}
& P_{Djk} \big( \rho_{n(j)} - \lambda_{D_{jk}}^u  - \sigma_j + \psi_{jk} \big) = 0 \label{Dopt4eq} \\
& \nonumber \forall  j\in D_q;k=1, \ldots,N_{Dj}\\
& \psi_j  \big( \sum_{k=1}^{N_{Dj}}P_{Djk} - P_{Djk}^{min} \big) = 0 ; \forall j\in D_q;k=1, \ldots,N_{Dj}  \label{Dopt5eq}\\
&\sigma_{jk} \big( -P_{Djk} + P_{Djk}^{max} \big) = 0 ; \forall j\in D_q;k=1, \ldots,N_{Dj}  \label{Dopt6eq}
\end{align}
\end{subequations}
\begin{defini}
\label{RCONCO}
ConCo $j$ is said to behave rationally if the associated linear operating cost denoted as $\lambda_{D_{jk}}^u$ is non-increasing with respect to $P_{D_{jk}}$. That is
\begin{equation}
d(\lambda_{D_{jk}}^u)/d (P_{D_{jk}}) <0, \; \forall j\in D_q,  k=1, \ldots,N_{Dj}.
\label{cond2}
\end{equation}
\end{defini}
\subsection{Effect of Uncertainties and Demand Response}
The most dominant impact of the introduction of distributed energy resources is uncertainties, which can directly alter the power generated. This in turn can affect the overall market equilibrium as well as the corresponding LMP. 
We take a first step in this direction by introducing uncertainties in the decision variables introduced in section A. We first separate the family of $P_{Gib}$ into $P^C_{Gib}$, $ i=1,\ldots,n_C$, and $P^R_{Glb}$, $l=1,\ldots,n_R$, where $n_C$ denotes the more conventional dispatchable generating units, and $n_R$ denotes distributed energy resources such as those based on wind and solar energy, which are non-dispatchable. 
At each node, both conventional generators and wind generators may be committed for production. We assume that the wind GenCo are competitive and that they submit their bids to the market as other conventional GenCo, and not modeled as a negative demand. It is also assumed that the available wind energy is overestimated which implies that if the assumed power is not available, power can be purchased from an alternate source or that loads can be shed. Underestimation implies that surplus power is either sold to adjacent utilities, or consumed through fast redispatch and automatic gain control, or reduced through reduction of conventional generation. This case is however not discussed in this paper. 

Using the above discussion, the objective function defined in \rep{gencooverall} is modified as
\begin{equation}
maximize \;  \sum_{l \in G_w}\sum_{b=1}^{N_{Gl}} \big( \rho_{n(l)}-\lambda_{G^w_{lb}}^C - C^r_{w_{lb}}( \Delta_{w_{lb}} )   \big) P^w_{Glb}
\label{windoverall}
\end{equation}
where
$P^w_{Glb}$ is the wind power from the $l^{th}$ wind generator, and assumed to lie in the interval $\left[ P^w_{Glb_{min}}, P^w_{Glb_{max}}\right]$. $ \Delta_{w_{lb}}$ is due to wind uncertainty, given by
\begin{equation}
 \Delta_{w_{lb}} = \bar{P}^w_{Glb} \Delta_{Glb} 
\label{uncergenco}
\end{equation}
where 
$\bar{P}^w_{Glb}$ is the mean wind power,
and $0<\Delta_{Glb}<1.$ 
We note that if there is underestimation, then $\Delta_{Glb}$ needs to be negative.

$\lambda_{G^w_{lb}}^C$ is the marginal cost function for the $l^{th}$ wind generator which is very close to zero.
Finally, $C^r_{w_{lb}}(\Delta_{w_{lb}})$ is a cost incurred by committing specific generators as reserves \cite{conejowind2009}, due to the wind uncertainty $\Delta_{w_{lb}}$, and is modeled as a quadratic function
\begin{equation}
C^r_{w_{lb}}(\Delta_{w_{lb}})  = b_{w_{lb}} \Delta_{w_{lb}} + \frac{c_{w_{lb}}}{2}\Delta_{w_{lb}}^2.
\label{windcost}
\end{equation}
We define another quantity $x^w$ which represents the percentage of wind penetration as
\begin{equation}
\begin{split}
x^w = \frac{\sum_{l \in G_w}\sum_{b=1}^{N_{Gl}} P^w_{Glb} }{  \sum_{j \in D_q}\sum_{k=1}^{N_{Dj}}P_{Djk}^{max}}
\end{split}
\label{pen}
\end{equation}
where $P_{Djk}^{max}$ is the maximum power demanded by consumer $j$ in block $k$ defined in \rep{maxd}. We note that  the impact of wind power on the market equilibrium is much smaller if $x^w$ is small, i.e., if wind penetration is low, than if $x^w$ is large. 

%


Yet another component that we introduce into the picture is Demand Response \cite{Yang2009, Bompard2007}. As for GenCo, categorizing all ConCo units into dispatchable and non-dispatchable ones, we consider a dispatchable load $P_{Djk}$ and assume that it is affected by Demand Response and is sensitive to price changes. This effect is modeled using a control parameter $\kappa_{Djk}$, and denotes the response of the consumers to a change in the Real Time Price (RTP) which occurs due to Demand Response.
\begin{equation}
\bar{P}_{Djk} = P_{Djk} \big( 1-\kappa_{Djk} \big) \qquad 0 < \kappa_{Djk} < 1
\nomenclature{$ \kappa_{Djk} $}{Curtailment Factor of block $k$ by demand $j$}
\label{uncerconco}
\end{equation}
$\bar{P}_{Djk}$ denotes the consumption incorporated with demand responsiveness into RTP. A positive $\kappa_{Djk}$, denotes a decrease in the ConCo consumption while a negative $\kappa_{Djk}$, denotes an increase. 
In this paper, we restrict our attention to positive $\kappa_{Djk}$ since our focus is on cases where there is a shortfall in the non-dispatchable GenCo, i.e. $\Delta_{Glb}>0$.
 It is assumed that $\kappa_{Djk}$ is suitably calibrated to represent the effect of RTP on the consumer behavior. Noting that this control parameter is proportional to the elasticity factor $\epsilon$ defined as \cite{Bompard2007}
\begin{equation}
\epsilon = \dfrac{\frac{\Delta P_{Djk}}{P_{Djk}}}{\frac{\Delta \rho_{n(j)}}{\rho_{n(j)}}}
\nomenclature{$ \Delta P_{Djk} $}{Change of consumption by demand $j$}
\nomenclature{$ \Delta \rho_{n(j)} $}{Change of LMP for demand $j$}
\end{equation}
and is a measure of the consumption response to RTP changes, we denote this control parameter, $\kappa_{Djk}$ as Curtailment Factor.

In contrast to $\Delta_{Gkb}$, which may be unknown, curtailment factor $\kappa_{Djk}$, is assumed to be controllable. Nevertheless, the equilibrium as well as the LMPs are altered due to the presence of the perturbation parameters $\Delta_{Gkb}$, and $\kappa_{Djk}$. In Section III, a similar optimization procedure as above and equivalent KKT conditions to \rep{Gopt1eq}-\rep{Gopt3eq} and \rep{Dopt4eq}-\rep{Dopt6eq} are derived, and the dependence of the equilibrium and LMPs on the perturbation parameters are discussed in detail.

It should be noted that the effect of RTP is modeled in this paper in the form of  $\kappa_{Djk}$.The inherent assumption here is that the nondispatchable ConCo observes the LMP, which is the solution of \rep{Dopt4eq}-\rep{Dopt6eq}, and suitably adjusts its demand. This adjustment is represented through the curtailment factor  $\kappa_{Djk}$ and therefore the elasticity factor $\epsilon$.  Therefore, it may be argued that the RTP is assumed to be the same as LMP in this paper, and that the effect of RTP is modeled essentially as a Demand Response parameter, denoted by $\kappa_{Djk}$. Further improvements of such a modeling of RTP and demand response are the focus of ongoing research.
\subsection{Independent System Operator ISO modeling}
In addition to GenCo and ConCo, a competitive electricity market also includes a coordinator which is independent of these profit-maximization entities \cite{shahidehpour2002, shahidehpour2001, Conejo2002}. Such an independent entity is necessary in order to maintain order in the overall market. Denoted as Independent System Operator (ISO), this entity is responsible for guaranteeing a non-discriminating access for all participants to the market, and enforces various guidelines such as maintaining transmission capacity limits.
While in practice, distinctions are made between an ISO and what is termed a Market Operator (MO) where the latter is responsible for financial management and the latter for overall physical supervision, in this paper we combine the two roles into the same rubric of an ISO. 
The responsibilities of the ISO are
to enforce transmission capacity limits, to maintain independence from the
market participants, to avoid discrimination against the market participants,
and to promote the efficiency of the market.

The electricity market that we consider in this paper is wholesale and is assumed to function as follows: First, each generating
company submits the bidding stacks of each of its units to the pool. Similarly, each
consumer submits the bidding stacks of each of its demands to the pool. Then,
the ISO clears the market using an appropriate market-clearing procedure
resulting in prices and production and consumption schedules.
The market-clearing procedure may introduce network constraints, which
model losses \cite{shahidehpour2002} and line capacity limits \cite{illic1998}. In
this paper we include network constraints in the underlying market model and the resulting prices are therefore
the Locational Marginal Prices \cite{Schweppe1988}. This implies that a generating unit injecting
power at a given node is paid the locational marginal price corresponding to
that node; and conversely, a demand receiving power from a given node pays
the locational marginal price corresponding to that node.
The most dominant network constraints are due to line capacity limits and network losses. Technical constraints cause the power flow through any line to be limited. The power flow is said to be congested when it approaches its maximum limit. This constraint is explicitly included in our model below. The second constraint is due to losses, most of which are due to the heat loss in the power lines. For ease of exposition, such ohmic losses are not modeled in this paper. 
The standard market-clearing procedure is based on Social Welfare \cite{shahidehpour2002}. Denoting this as $S_W$, it is defined as 
\begin{equation}
S_W=\sum_{j \in D_q}\sum_{k=1}^{N_{Dj}}  \lambda_{Djk}^B P_{Djk}- \sum_{i \in G_f}\sum_{b=1}^{N_{Gi}}\lambda_{G_{ib}}^B P_{Gib}
\label{sw}
\end{equation}
where the first and second term denote the revenue due to surpluses stemming from bids from GenCo and ConCo, respectively.
The market-clearing procedure is then given by
\begin{equation}
Maximize \; S_W 
\label{ISO}
\end{equation}
subject to
\begin{subequations}
\begin{align}
& \nonumber  \sum_{i \in \theta}\sum_{b=1}^{N_{Gi}}P_{Gib} -  \sum_{j \in \vartheta}\sum_{k=1}^{N_{Dj}} P_{Djk} -  \\
&\sum_{m \in \Omega} B_{nm}\big[ \delta_n-\delta_m \big]=0  \label{balance}  ;\rho_{n} \forall n \in N \\
& B_{nm}\big[ \delta_n-\delta_m\big]  \leq P_{nm}^{max} ;  \gamma_{nm} \forall n \in N; \forall m \in \Omega \label{capacity}
\end{align}
\end{subequations}
The constraints \rep{balance} and \rep{capacity} are due to power balance and capacity limits, respectively.
It can be seen that the associated Lagrange multipliers, $ \rho_n $ and $ \gamma_{nm}$, are indicated in each constraint.

The underlying optimization problem of the ISO can be therefore defined as the optimization of \rep{ISO} subject to constraints \rep{balance} and \rep{capacity}. 
This problem can be restated as the solutions of the generation power blocks levels $P_{Gib}$, the
demand power blocks levels $P_{Djk}$, the voltage angle $\delta_n$, and dual variables $ \rho_n $ and $ \gamma_{nm}$ such
that
\begin{subequations}
\begin{align}
& P_{Gib}\big( \lambda_{Gib}^B -  \rho_{n(i)}  \big) = 0 \;  \forall i \in G_f;b=1, \ldots,N_{Gi} \label{opt7eq} \\
& P_{Djk}\big(\rho_{n(j)} - \lambda_{Djk}^B \big) = 0 \;  \forall j \in D_q;k=1, \ldots,N_{Dj} \label{opt8eq} \\
& \nonumber \delta_n \big( \sum_{m \in \Omega} B_{nm}\big[ \rho_n-\rho_m \big] + \\
& \sum_{m \in \Omega} B_{nm} \big[ \gamma_{nm}-\gamma_{mn}\big] \big) = 0 \;  \forall n \in N \label{opt9eq} \\
& \nonumber \rho_n \big( -\sum_{i \in \theta}\sum_{b=1}^{N_{Gi}}P_{Gib} + \sum_{j \in \vartheta}\sum_{k=1}^{N_{Dj}} P_{Djk} + \\
& \sum_{m \in \Omega} B_{nm} \big[ \delta_n-\delta_m \big] \big) =0 \; \forall n \in N; \rho_n > 0 \label{opt10eq1} \\
& \gamma_{nm} \big( P_{nm}^{max} - B_{nm} \big[ \delta_n-\delta_m \big]  \big)= 0\; \forall n \in N; \forall m \in\Omega \label{opt11eq} 
\end{align}
\end{subequations}
The decision variables of this problem are the amounts of power to be
generated by each generating unit $i$ in each block $b$, i.e.,  $P_{Gib}$; the amounts of
power to be consumed by each demand $j$ in each block $k$, i.e., $P_{Djk}$ and
the locational marginal prices, $\rho_{n(j)}$.
It should be noted that \rep{opt10eq1} is arrived at using the balance equation, which implies that the term in the corresponding parenthesis is zero, and by multiplying this term by $\rho_n$, which is positive.

In summary, the market model that we consider in this paper include optimization goals of GenCo discussed in section II-A. and formulated in eqs. \rep{gencooverall}-\rep{posgeneach}, optimization goals of ConCo, \rep{objd}-\rep{nonzerod}, and formulated in II-B, and finally optimization of the ISO discussed in section II-D, and formulated in eqs \rep{sw}-\rep{capacity}.  The combined problem statement is the determination of an equilibrium point which collectively optimizes all the above three market contributors.
The corresponding equilibrium point is therefore defined as that which satisfies the following criteria:
 \begin{enumerate}
\item maximum profit for every individual generating company
\item maximum utility for every individual consumer
\item maximum net social welfare for the ISO.
 \end{enumerate}
\section{Market Equilibrium}
\label{equi}
We now discuss the overall market equilibrium. In the GenCo and ConCo optimization problems, given in  \rep{gencooverall}-\rep{posgeneach} and \rep{objd}-\rep{nonzerod}, respectively, the Locational Marginal Prices appear as inputs. On the other hand, these LMPs appear in the optimization problem of the ISO in \rep{sw}-\rep{capacity} as dual prices corresponding to the balance equations, causing a tight link between the three families of optimization problems, thereby making the overall problem nontrivial. 

We note that  each of these three sets of optimization
problems are linear programming problems. Thus, the Karush-Kuhn-Tucker
(KKT) optimality conditions are both necessary and sufficient for describing the optimal solutions. The optimality conditions of the three sets of problems result in a Linear Complementarity Problem (LCP) which correspond to the market equilibrium for all $i\in G_f$, $b=1, \ldots,N_{Gi}$, $j\in D_q$, $ k=1, \ldots,N_{Dj}$, $ n \in N$, and $m \in \Omega$ as follows:
\begin{subequations}
\begin{align}
& P_{Gib}\big(-\rho_{n(i)} + \lambda_{G_{ib}}^C  + \alpha_i + \phi_{ib} \big) = 0  \label{opt1eq} \\
& \alpha_i \big( -\sum_{b=1}^{N_{Gi}}P_{Gib} + P_{Gi}^{max}\big) = 0 \label{opt2eq} \\
& \phi_{ib} \big( P_{Gib}^{max} - P_{Gib} \big) = 0 \label{opt3eq} \\
& P_{Djk} \big( \rho_{n(j)} - \lambda_{D_{jk}}^u  - \sigma_j + \psi_{jk} \big) = 0  \label{opt4eq} \\
& \psi_j  \big( \sum_{k=1}^{N_{Dj}}P_{Djk} - P_{Djk}^{min} \big) = 0  \label{opt5eq}\\
& \sigma_{jk} \big( -P_{Djk} + P_{Djk}^{max} \big) = 0   \label{opt6eq} \\
& P_{Gib}\left( \lambda_{Gib}^B -  \rho_{n(i)}  \right) = 0  \label{opt7eq} \\
& P_{Djk}\left(\rho_{n(j)}  - \lambda_{Djk}^B\right) = 0  \label{opt8eq} \\
&\nonumber  \delta_n \big( \sum_{m \in \Omega} B_{nm}\big[ \rho_n-\rho_m \big] + \\
& \sum_{m \in \Omega} B_{nm}\big[ \gamma_{nm}-\gamma_{mn}\big] \big) = 0  \label{opt9eq} 
\end{align}
\begin{align}
&\nonumber \rho_n \big(-\sum_{i \in \theta}\sum_{b=1}^{N_{Gi}}P_{Gib} + \sum_{j \in \vartheta}\sum_{k=1}^{N_{Dj}} P_{Djk} +\\
& \sum_{m \in \Omega} B_{nm} \big[ \delta_n-\delta_m \big] \big) =0 \label{opt10eq} 
\end{align}
\begin{align}
& \gamma_{nm} \big( P_{nm}^{max} - B_{nm} \big[ \delta_n-\delta_m \big]  \big) = 0 \label{opt11eq} 
\end{align}
\end{subequations}
The above problem can be compactly stated as the solution of the following LCP: Find a vector $x^* \in R^n$ that solves the following constraints:
\begin{align}
& \nonumber x^T \big( Mx + q \big) = 0 \\
& \nonumber x \geq 0 \\
& Mx + q \geq 0
\label{mlcp}
\end{align}
where $x \in R^{n \times 1}$ is the variables vector and square matrix $M \in R^{n \times n}$ and $q \in R^{n \times 1}$ contains operating costs and bids of generators, utilities and bids of consumers, maximum and minimum limit of generators and consumers and also maximum thermal limit of transmission lines (See Nomenclature for all definitions). The corresponding solution $x^*$ of the underlying LCP problem, which determines the market equilibrium and is dependent on $M$ and $q$, is denoted as 
\begin{equation}
x^* \triangleq LCP(M,q).
\label{xdef}
\end{equation}
\begin{defini}
A matrix $M \in R^{n \times n}$ is called a P-matrix if its all principal minors are positive.
\end{defini}

We now state and prove the theorem that discusses the LCP solution.
\begin{thm}
In LCP problem defined in \rep{mlcp}, matrix $M$ is  a P-matrix if and only if the $LCP(M,q)$ has a unique solution for any $q \in R^n$. Moreover, if $M$ is a P-matrix then there is a neighborhood 
${\cal M}$ of $M$, such that all matrices in ${\cal M}$ are P-matrix. 
\label{more}
\end{thm}
\begin{proof}
see \cite{More1997}.
\end{proof}
An additional result that will prove to be useful for our following discussions concerns the dependence of $x(M,q)$ on the market parameters. This is characterized in the following theorem.
\begin{thm} \cite{Chen2007}
Matrix $A$ is a P-matrix if and only if $\left( I-D+DA \right) $ is nonsingular for any diagonal matrix $D=diag (d)$ with $0 \leq d_i \leq 1$, $i= 1,2, \ldots, n$. Furthermore, if $x(A,q)$ and $x(B,p)$ are the solution of the corresponding $LCP(A,q)$ and $LCP(B,p)$ respectively, we have:
\begin{equation}
\frac{\parallel x(B,P) - x(A,q)\parallel}{\parallel x(B,P) \parallel} \leq \beta(A) \parallel (A-B)x(B,p)+q-p \parallel
\label{lemma1}
\end{equation}
where 
\begin{equation}
\beta(A) = max_{d \in [0,1]} ||(I-D+DA)^{-1}D||.
\label{beta}
\end{equation}
\end{thm}

In the following theorem, the connection between market equilibrium and Pure Strategy Nash Equilibrium is presented.
\begin{thm}
\label{5}
Assume that all GenCos and ConCos are rational players. If $M$ is a P-matrix,
then equilibrium of the market denoted as $x^* \triangleq LCP(M,q)$ exists and is identical to  a unique Pure Strategy Nash Equilibrium.
\end{thm}
\begin{proof}
We prove this proposition by using contradictions. That is, assume that GenCos and ConCos are rational, $M$ is a P-matrix and $x^* \triangleq LCP(M,q)$ is not a Pure Strategy Nash Equilibrium. 
Since $M$ is a P-matrix, then Theorem \ref{more} implies that  $x^* \triangleq LCP(M,q)$ exists and is the unique maximizer of the GenCo problem denoted in \rep{gencooverall} - \rep{maxgeneach}, ConCo problem denoted by \rep{objd} - \rep{mind}, and ISO problem denoted by \rep{ISO} - \rep{capacity}. That is  
\begin{equation} 
u_i(x_i,x^*_{-i}) \leq u_i(x^*_{i},x^*_{-i}) \; \forall x_{i} \in S_{i}.
\label{NE3}
\end{equation}

Since GenCos and ConCos are rational, the inequalities in \rep{cond1} and \rep{cond2} hold. From Theorem 2, it follows that the payoff functions $(u_1,...,u_I)$ are diagonally strictly concave. Theorem \ref{thm2}, in turn, implies that the game between GenCo, ConCo, and ISO has a unique Pure Strategy Nash equilibrium denoted as ${x_N}^*$. Using the Nash Equilibrium definition in \rep{NE}, it follows that 
\begin{equation} 
u_i(x_i,x^*_{-i}) \leq u_i(x^*_{N_i},x^*_{N_{-i}})\; \forall x_{i} \in S_{i}.
\label{NE2}
\end{equation}
Since we assume that $x^* \triangleq LCP(M,q)$ is not a Pure Strategy Nash Equilibrium, it implies that both $x^*$ and $x^*_N$ maximize the payoffs which contradicts the uniqueness of $x^* \triangleq LCP(M,q)$ and therefore completes the proof. 
\end{proof}
\begin{remark}
Theorem \ref{5} provides conditions under which the market has a unique equilibrium, i.e. no degeneracy can occur \cite{conejo2006} in the market if $M$ is a P-matrix.
\end{remark}

We now proceed to the effect of uncertainties and curtailment factor on the market equilibrium.
\section{Perturbation Analysis of Market Equilibrium}
\label{pert}
Using the results in the previous section,  we now present a perturbation bound for the equilibrium of the electrcial market.
As we discussed above, the market equilibrium is the solution of $LCP(M,q)$ which is denoted as $x^* $ and has been explicitly defined in Eq. \rep{mlcp}.  
As we discussed in Section II.C, the wind forecast error defined in \rep{uncergenco}, parametrized by  $\Delta_{Glb}$, and demand response defined in \rep{uncerconco}, parametrized by $\kappa_{Djk}$, are considered as two sources of perturbations. These two components of uncertainty can affect $M$ and $q$ which explicitly in our analysis are denoted by $M+\Delta M$ and $q+\Delta q$, respectively. Therefore the underlying LCP problem under parametric perturbation is altered as 
\begin{align}
& \nonumber x^T \left( (M+\Delta M) x + (q+\Delta q) \right) = 0 \\
& \nonumber x \geq 0 \\
& (M+\Delta M) x + (q+\Delta q) \geq 0
\label{mlcppert}
\end{align}
This in turn leads to a corresponding equilibrium
\begin{equation}
x_\Delta^* \triangleq  LCP(M+\Delta M, q+\Delta q).
\label{xperdef}
\end{equation} 
In the rest of the paper we are looking for the maximum shift of the equilibrium.
The following definitions are useful.
\begin{defini}
Define non-dimensional perturbation parameters $\epsilon_M$ and $\epsilon_q$ as
\begin{equation}
||\Delta M|| \leq \epsilon_M \parallel M \parallel ,
\end{equation}
\begin{equation}
||\Delta q|| \leq \epsilon_q \parallel q \parallel ,
\end{equation}
a constant $\eta$ as
\begin{equation}
\eta=\epsilon_M \beta (M) \parallel M \parallel ,
\end{equation}
\end{defini}
and a set ${\cal M}$ as
\begin{equation}
{\cal M}: \big\{  \Delta M \big| \beta(M)||\Delta M|| \leq \eta \big\} 
\label{setm}
\end{equation}

We now quantify the relation between $x^*_\Delta$ in \rep{xperdef} and $x^*$ in \rep{xdef}, in the following theorem.
\begin{thm}
\label{mthm}
If the nominal market \rep{mlcp} has unique solution and $\eta < 1$, then the perturbed market in \rep{mlcppert} has a unique solution $x_\Delta^*$ and satisfies the following inequality 
\begin{equation}
\frac{\parallel   x^* - x_\Delta^* \parallel}{\parallel   x^*  \parallel} \leq \frac{2 \epsilon}{1 - \eta} \beta (M)
\label{thm}
\end{equation}
\end{thm}
\begin{proof}
If the nominal electrical market \rep{mlcp} has a unique solution according to Theorem 1, the matrix M is P-matrix. This in turn, according to Theorem 2, implies that $\left( I-D+DM \right) $ is nonsingular for any diagonal matrix $D=diag (d_i)$ with $0 \leq d_i \leq 1$, $i= 1,2, \ldots, n$. We have the following equality
\begin{equation}
\begin{split}
\big( I-D+D(M+ \Delta M) \big) = \left( I-D+DM \right) \left(I+M_0 \Delta M \right) 
\end{split}
\label{ns}
\end{equation}
where
\begin{equation}
M_0 = (I-D+DM)^{-1}D.
\label{m0}
\end{equation}
Noting the definition of $\beta(M)$, it follows that 
\begin{align}
& ||M_0 \Delta M|| \leq \beta(M)||\Delta M||  \leq \eta \;  \forall  \Delta M \in {\cal M} 
\end{align}
Since $\eta < 1$ according to theorem 3, it follows that $I+ M_0 \Delta M$ is nonsingular for all $\Delta M \in {\cal M}$. We therefore conclude from \rep{ns} that $\big( I-D+D(M+ \Delta M) \big) $ is a P-matrix and therefore the market in \rep{mlcppert} has a unique solution for all $  \Delta M \in {\cal M} $ and $\eta <1$.

We now show that $x^{*}_{\Delta}$ satisfies \rep{thm}. Using \rep{lemma1} in Theorem 2, we have that:
\begin{equation}
\begin{split}
\frac{\parallel x^*  - x_\Delta^*   \parallel}{\parallel   x^* \parallel}  \leq 
\beta(M+\Delta M) ||\Delta M x^*   + \Delta q||
\label{upbound}
\end{split}
\end{equation}
We rewrite the argument of  $\beta(M+\Delta M)$ as
\begin{align}
&\nonumber \big( I-D+D(M+ \Delta M) \big)^{-1}D= \\
&\big(I+M_0(\Delta M)\big) ^{-1} (I-D+DM)^{-1}D
\label{betaper}
\end{align}
Using Appendix 1 in \cite{kianiisgt2010}, we have:
\begin{equation}
\begin{split}
|| \big(I+M_0(\Delta M)\big) ^{-1} || \leq
 \frac{1}{1-\beta (M) ||\Delta M||} \leq \frac{1}{1-\eta}
\end{split}
\end{equation}
Taking norms on both sides of \rep{betaper}, we obtain that:
\begin{equation}
\beta (M + \Delta M) \leq \frac{1}{1- \eta} \beta (M)
\label{upbeta}
\end{equation}
Using  \rep{upbound} implies that:
\begin{equation}
\begin{split}
\frac{ \parallel x^*  - x_\Delta^*  \parallel}{\parallel   x^* \parallel}  \leq 
 \frac{1}{1- \eta} \beta (M) ||\Delta M   +\Delta q ||
\end{split}
\end{equation}
Considering  the definition of $\Delta M$ and $\Delta q$, we obtain, in turn,
\begin{equation}
\begin{split}
\frac{ \parallel x^*  - x_\Delta^*  \parallel}{\parallel   x^* \parallel}  \leq \frac{2 \epsilon}{1- \eta} \beta (M)
\end{split}
\end{equation}
where $\epsilon = max \{ \epsilon_M \parallel M \parallel, \epsilon_q \parallel q \parallel \}$ which is the desired bound.
\end{proof}
\begin{remark}
\label{remark1}
Defining $\mu(\Delta_{Glb},\kappa_{Djk})$ for all $l\in G_w$, $b=1, \ldots,N_{Gi}$, $j\in D_q$, $ k=1, \ldots,N_{Dj}$ as 
\begin{equation} 
\mu(\Delta_{Glb},\kappa_{Djk}) = \frac{2 \epsilon}{1- \eta} \beta (M),
\label{mu}
\end{equation}
Theorem \ref{mthm} implies that the uncertainty in market can lead to a maximum shift by the equilibrium of an amount $\mu(\Delta_{Glb},\kappa_{Djk})$. As this function is nonlinear, determining an analytical relationship between $\mu(\Delta_{Glb},\kappa_{Djk})$,  $\Delta_{Glb}$ and $\kappa_{Djk}$ is exceedingly difficult. As will be shown in the next section, simulation studies show that as $\kappa_{Djk}$ increases, $\mu$ decreases. This in turn brings the perturbed equilibrium closer to the nominal equilibrium.
\end{remark}

\subsection{Game theoretic Interpretations}
We now evaluate the perturbed market using tools from game theory.
Consider two strategic games defined by two profiles of utility functions as $ G = \langle \mathcal{I}, (S_i)_{i \in \mathcal{I}} , (u_i)_{i \in \mathcal{I}} \rangle$ and $ \tilde{G} = \langle \mathcal{I}, (S_i)_{i \in \mathcal{I}} , (\tilde{u}_i)_{i \in \mathcal{I}} \rangle$. If $s^*_i$ is a Nash equilibrium of $G$, then $s^*_i$ need not be a Pure Strategy Nash Equilibrium of $\tilde{G}$. The equilibria of the strategic games $G$ and $\tilde{G}$ may be far apart, even if $(u_i)_{i \in \mathcal{I}}$ and $(\tilde{u}_i)_{i \in \mathcal{I}}$ are very close to each other. 
\begin{defini}
\label{epsilon}
Given $\epsilon \geq 0$, a Pure Strategy Nash Equilibrium $s^* \in S$ is called an $\epsilon-\text{equilibrium}$ if for all $i \in \mathcal{I}$ and $s_i \in S_i$,
\begin{equation}
u_i(s_i,s^*_{-i}) \leq u_i(s^*_i,s^*_{-i})+ \epsilon  \; \forall s_{i} \in S_{i}.
\end{equation}
\end{defini}
Obviously in Definition \ref{epsilon}, when $\epsilon = 0$, the $\epsilon-\text{equilibrium}$ is a Pure Strategy Nash Equilibrium in the sense of Definition \ref{ne}. 

The following definition formally defines the closeness of two strategic form games.
\begin{defini}
\label{alpha}
Let $ G = \langle \mathcal{I}, (S_i)_{i \in \mathcal{I}} , (u_i)_{i \in \mathcal{I}} \rangle$ and $ \tilde{G} = \langle \mathcal{I}, (S_i)_{i \in \mathcal{I}} , (\tilde{u}_i)_{i \in \mathcal{I}} \rangle$ be two strategic form games, then $\tilde{G}$ is a $\alpha-\text{approximation}$ of $G$ if for all $i \in \mathcal{I}$,
\begin{equation}
|u_i(s) - \tilde{u}_i(s)| \leq \alpha \; \forall s \in S.
\end{equation}
\end{defini}
The next proposition relates the $\epsilon-\text{equilibrium}$ of close games as defined in \ref{alpha}.
\begin{thm}
\label{epsclos}
If $ \tilde{G} = \langle \mathcal{I}, (S_i)_{i \in \mathcal{I}} , (\tilde{u}_i)_{i \in \mathcal{I}} \rangle$ is an $\alpha-\text{approximation}$ to $G = \langle \mathcal{I}, (S_i)_{i \in \mathcal{I}} , (u_i)_{i \in \mathcal{I}} \rangle$ and $s^* \in S$
is an equilibrium 
of $ \tilde{G}$, then $s^* \in S$ is a $(2\alpha)-\text{equilibrium}$ of $G$.
\end{thm}
\begin{proof}
For all  $i \in \mathcal{I}$ and all $s_i \in S_i$, we can write
\begin{equation}
\begin{aligned}
u_i(s_i,s^*_{-i}) - u_i(s^*_i,s^*_{-i}) = u_i(s_i,s^*_{-i}) -  \tilde{u}_i(s_i,s^*_{-i}) + \\
\tilde{u}_i(s_i,s^*_{-i})  -\tilde{u}_i(s^*_i,s^*_{-i}) + \\
\tilde{u}_i(s^*_i,s^*_{-i}) - u_i(s^*,s^*_{-i}).
\end{aligned}
\end{equation}
Since $s^*$ is a Pure Nash Equilibrium of $\tilde{G}$, then $\tilde{u}_i(s_i,s^*_{-i})  -\tilde{u}_i(s^*_i,s^*_{-i}) \leq 0$. This in turn follows that 
\begin{equation}
\begin{aligned}
u_i(s_i,s^*_{-i}) - u_i(s^*_i,s^*_{-i}) \leq |u_i(s_i,s^*_{-i}) -  \tilde{u}_i(s_i,s^*_{-i})| + \\
|\tilde{u}_i(s^*_i,s^*_{-i}) - u_i(s^*,s^*_{-i})|.
\end{aligned}
\end{equation}
Using Definitions \ref{alpha} and \ref{epsilon}, it follows that
\begin{equation}
\begin{aligned}
u_i(s_i,s^*_{-i}) - u_i(s^*_i,s^*_{-i}) \leq 2\alpha
\end{aligned}
\end{equation}
which in turn implies that $s^* \in S$ is an $(2\alpha)-\text{equilibrium}$ of $G$.
\end{proof}
\begin{pro}
\label{lem1}
Assume that the perturbed market denoted by the strategic game $\tilde{G}$ under uncertainties $\Delta_{Glb}$ and $\kappa_{Djk}$ is an $\alpha-\text{approximation}$ of the nominal market denoted by $G$. If $x_\Delta^*$ is an $\text{equilibrium}$ of the perturbed market $\tilde{G}$, then $x_\Delta^*$ is an $(2\alpha)-\text{equilibrium}$ of the nominal market $G$.
\end{pro}
\begin{remark}
If the wind forecast error $\Delta_{Glb}$ increases, from Definition  \ref{alpha}, it follows that $\alpha$ will increase as the relative payoffs of the corresponding generators will have an increased error. Proposition \ref{lem1} implies that in such a case, the equilibrium of the perturbed
game $\tilde{G}$ is correspondingly far away from the nominal market $G$. It should be noted that as the forecast error increases, the corresponding cost of deploying ancillary services increases as well. If a Demand Response program is in place, this cost increase is conveyed to the consumer, leading to a decrease in the load quantified by the demand curtail factor $\kappa_{Djk}$ as in \rep{uncerconco}. Similar to our observation in Remark \ref{remark1}, it should be noted that the exact impact of an increasing  $\kappa_{Djk}$ on $\alpha$ is difficult to quantify due to its nonlinearity. However, simulation studies show, as discussed in Section \ref{sim}, that as $\kappa_{Djk}$ increases, $\alpha$ decreases. This in turn brings the perturbed equilibrium closer to the nominal equilibrium.
\end{remark}

\section{Case Study}
\label{sim}
An IEEE 30-bus case is used for simulation studies, whose interconnections are shown in Figure \ref{fig:market}. The size and price of each block of each GenCo are shown in Table  \ref{table_property1} and for the sake of simplicity we assume that each GenCo bids in its marginal cost.
The size of minimum power requirement of each demand and the size and price of each block of each demand are shown in Table \ref{table_property2}. We assume that only part of the overall loads are participating in the market, and the fixed loads in other buses are shown in Table \ref{fixed-load}.
The reactance $B_{nm}$ of the line connecting bus $n$ and bus $m$ can be found in Table \ref{table_property3}. The transmission capacity limit of  all lines is chosen to be $100$$MW$. The line parameters are perunit with three-phase base of 230 kV and 10 MVA.
\begin{table}[btp]
\caption{Cost functions data of Generators}
\label{table_property1}
\begin{center}
\scalebox{0.9} 
{\begin{tabular}{|c|c c|c c|}
\hline
Name&  Block 1& & Block 2& \\
& Size$^{(MW)}$& Price$^{(\$/ MW)}$& Size$^{(MW)}$& Price$^{(\$/ MW)}$\\
\hline
$P_{g_1}$& 12.5 & 240 & 40 & 900 \\\hline
$P_{g_2}$&  12.7 & 300 & 40  & 950  \\\hline
$P_{g_5}$&  12.2 & 240 & 80  & 1200  \\\hline
$P_{g_{11}}$& 12.2  & 240 & 80  & 1200  \\\hline
$P_{g_{13}}$& 	10 & 40 & 20  & 80  \\
\hline
\end{tabular}}
\end{center}
\end{table}

\begin{table}[btp]
\caption{Cost functions data of Consumers}
\label{table_property2}
\begin{center}
{\begin{tabular}{|c|c c|c c|}
\hline
Name&  Block 1& & Block 2& \\
& Size$^{(MW)}$& Price$^{(\$/ MW)}$& Size$^{(MW)}$& Price$^{(\$/ MW)}$\\
\hline
$P_{d_7}$& 11 & 150 &  19 & 600 \\
  \hline 
$P_{d_{15}}$&10 & 100 &  21 & 600  \\
\hline
$P_{d_{30}}$&12 & 170 & 22  &  600 \\
\hline 
$P_{d_{9}}$&12 & 150 & 24  &  600 \\
\hline
$P_{d_{26}}$&15.5 & 150 & 21  &  600 \\
\hline
$P_{d_{27}}$&9.5 & 150 & 22  &  600 \\
\hline 
\end{tabular}}
\end{center}
\end{table}
\begin{table}[btp]
\caption{Fixed load data}
\label{fixed-load}
\begin{center}
{\begin{tabular}{|c|c ||c|c|}
\hline
Name&  Demand & Name & Demand \\
\hline
$P_{d_3}$& 22.31 & $P_{d_{14}}$ & 7.20 \\
 \hline 
$P_{d_4}$& 8.83& $P_{d_{16}}$ & 4.06 \\
\hline 
$P_{d_8}$& 13.95&$P_{d_{17}}$ & 10.46 \\
 \hline 
$P_{d_{10}}$& 6.74& $P_{d_{18}}$ & 3.72 \\ 
\hline
$P_{d_{12}}$& 13.01 & $P_{d_{19}}$ & 11.04 \\ 
\hline 
$P_{d_{20}}$& 2.55 & $P_{d_{21}}$ & 3.39 \\ 
\hline
$P_{d_{23}}$& 22.31 & $P_{d_{24}}$ & 10.11 \\ 
\hline
$P_{d_{26}}$& 4.06 & $P_{d_{29}}$ & 2.78 \\ 
\hline
\end{tabular}}
\end{center}
\end{table}

\begin{table}[]
\caption{Transmission Lines Reactance; Line reactances are perunited based on $100 MW$}
\label{table_property3}
\begin{center}
\begin{tabular}{|c|c|c||c|c|c|}
\hline
 \multicolumn{2}{|c|}{Connected} & \multirow{2}*[-1.2ex]{Reactance} & \multicolumn{2}{|c|}{Connected} & \multirow{2}*[-1.2ex]{Reactance}\\
 \multicolumn{2}{|c|}{Bus} & & \multicolumn{2}{|c|}{Bus} &\\\cline{1-2}\cline{4-5}
 Bus & Bus & \multirow{2}*[1.2ex]{($x_{b\_k}$)} &Bus & Bus & \multirow{2}*[1.2ex]{($x_{b\_k}$)}\\
 b & k &&b & k &\\\hline
 1 &	2& 0.0575 & 1 &	3&	0.1652 \\
2 &	4& 	0.1737 &3&	4&	0.0379 \\
2 &	5&	0.1983&2	&6	&0.1763\\
4 &	6&	0.0414& 5&	7&	0.116\\
6 &	7&	0.082&6&	8&	0.042 \\
6 &	9&	0.208&6&	10&	0.556\\
9 &	11&	0.208& 9&	10	&	0.11\\
4	&12&	0.256& 12&	13&	0.14\\
12 &	14&	0.2559& 12&	15&		0.1304\\
12 &	16 &	0.1987&14 &	15&	0.1997\\
16 &	17&		0.1923& 15 &18 &		0.2185\\
18&	19&0.1292& 19&	20&	0.068\\
10 &	20&	0.209&10&	17&0.0845\\
10	&21& 0.0749&10&	22&0.1499\\
21 &	22	&	0.0236&15&	23&	0.202\\
22 &	24&		0.179&23&	24&	0.27\\
24 &	25&		0.3292&25&	26&	0.38\\
25 &	27&		0.2087&28 &	27&	0.396\\
27 &	29&		0.4153&27	&30	&	0.6027\\
29&	30&		0.4533&8&	28&	0.2\\
6&	28&		0.0599&&&\\ \hline
 \end{tabular}
\end{center}
\end{table}
\begin{figure}[ht]
\centering	
\includegraphics[totalheight=10cm]{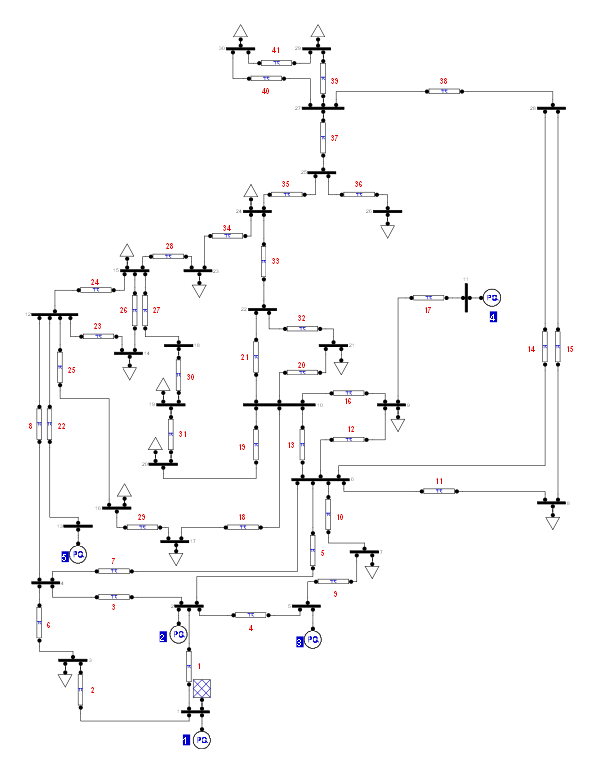}
\caption{IEEE 30-bus case study}
\label{fig:market}
\end{figure}
\subsection{Nominal Market Equilibrium}
Table \ref{table_resultGenCo} provides equilibrium results concerning generator output, revenue, and profit for the
electric power market if no wind shortfalls are imposed to the system. These results are
obtained by directly solving LCP problem \rep{opt1eq}-\rep{opt11eq}.
Table \ref{table_resultConCo} shows the power consumed and the corresponding demand payments. 
Since we ignore losses in this study, the difference between the total generator power output and the total power consumed, and are equal to zero and both are $221^{MW}$ whereas $151^{MW}$ of the load is a fixed demand as can be seen in Table \ref{fixed-load}.
The Locational Marginal Price in all of the buses is $23.62^{\$/MWh}$ due to the fact that transmission lines are not congested.
\begin{table}[btp]
\caption{Results of market equilibrium for GenCo, no wind uncertainty}
\label{table_resultGenCo}
\begin{center}
\scalebox{0.9} 
{\begin{tabular}{c| c c c c }
\hline
Name& Power output$^{MW}$ & Revenues$^{\$/h}$ & Cost$^{\$/h}$ & Total Profit$^{\$/h}$ \\
\hline
$P_{g_1}$& 30 &  708.6 &  660 &  48.6  \\\hline
$P_{g_2}$& 11.64 & 274.93 & 274.76 &  0.18    \\\hline
$P_{g_5}$&  80 & 1889.6 & 1000 & 889.6   \\\hline
$P_{g_{11}}$& 80  & 1889.6 & 1000 & 889.6  \\\hline
$P_{g_{13}}$& 20 & 472.4 &8 & 464.4  \\
\hline
\end{tabular}}
\end{center}
\end{table}

\begin{table}[btp]
\caption{Results of market equilibrium for ConCo, no wind uncertainty}
\label{table_resultConCo}
\begin{center}
{\begin{tabular}{c|c c c }
\hline
Name&  Power consumed & Total Payment $^{\$/h}$\\
\hline
$P_{d_7}$& 11& 259.82 \\
  \hline 
$P_{d_{15}}$ & 12 &  283.44 \\
\hline
$P_{d_{30}}$&10 &   236.2 \\
\hline 
$P_{d_{9}}$&15.5 &   366.1 \\
\hline
$P_{d_{26}}$&9.5 &  224.4 \\
\hline
$P_{d_{27}}$& 12.0 &  283.44 \\
\hline 
\end{tabular}}
\end{center}
\end{table}

\subsection{Perturbed Market Equilibrium with Wind Uncertainty}
We assumed that GenCo at bus $13$ in figure \ref{mu_uncertainty}, is wind based and is subjected to an uncertainty  $\Delta_{G13}$. For the sake of simplicity in simulations we assume that uncertainty for all blocks are the same and characterized by  $\Delta_{G13}$. The resulting perturbed market, discussed in Section \ref{pert}, was simulated.
Figure \ref{mu_uncertainty} shows the corresponding market equilibrium shift $\mu(\Delta_{Glb},\kappa_{Djk})$ as define in \rep{mu} for different wind penetration.

\begin{figure}[t]
\centering	
\includegraphics[totalheight=7cm]{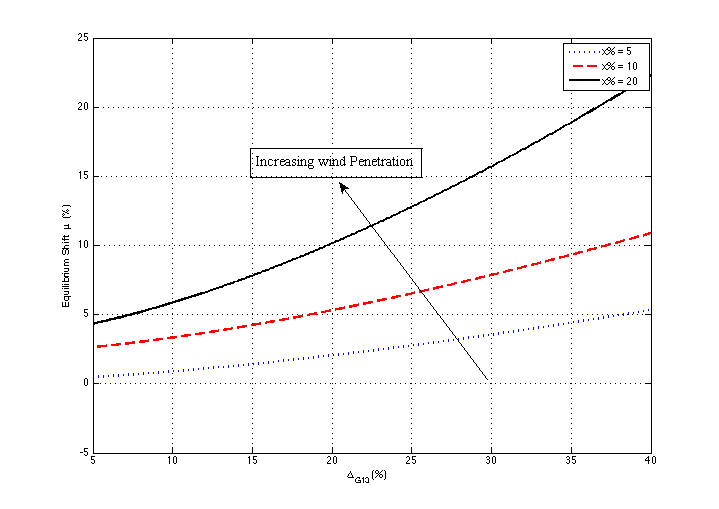}
\caption{Market equilibrium shift as a function of uncertainty  and wind penetration}
\label{mu_uncertainty}
\end{figure}

As can be seen in Fig \ref{mu_uncertainty}, with increasing $\Delta_{G13}$, higher equilibrium shift $\mu$ results. As wind penetration $x^w \%$ defined in \rep{pen} increases, the corresponding market equilibrium shift deviation is higher.

\subsection{Measure of Uncertainty incorporating Wind Uncertainty With Demand Response}
We now introduce a perturbation $\kappa_{D15}$ into the picture to represent demand response which correspond to introduce Real Time Pricing (RTP) at the consumer at bus $15$. We now numerically evaluate the effect of this perturbation parameter. In particular, we evaluate the equilibrium shift $\mu(\Delta_{Glb},\kappa_{Djk})$, for a range of $\kappa_{D15}$ at different level of wind forecast error in $\Delta_{G13}$. Figure \ref{mu_kapa} shows the corresponding results, which clearly demonstrate that RTP can reduce the equilibrium shift and therefore the cost of uncertainty in the presence of a wind forecast error.

\begin{figure}[t]
\centering	
\includegraphics[totalheight=6.7cm]{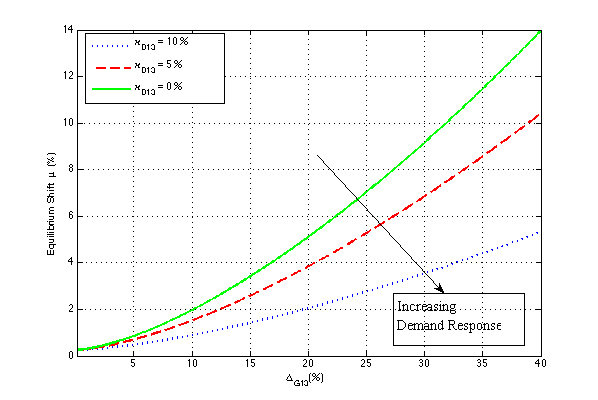}
\caption{Market equilibrium shift as a function of uncertainty and demand curtailment factor}
\label{mu_kapa}
\end{figure}

\section{Concluding Remarks}
\label{conc}
The current energy crisis has created an urgent need in integrating
renewable energy resources into the power grid. The latter in turn can introduce intermittency and uncertainty into the picture, thereby introducing a prohibitive integration cost. In this paper, we introduce an analytical framework to evaluate this cost. We begin with an overall model of the energy market including Generators Company, Consumers Company as well as Independent System Operator
(ISO). We analyze the underlying market equilibrium using game theory and established sufficient conditions are derived for the existence of a unique Pure Nash Equilibrium for the nominal market, then the effect of uncertainty due to renewable
energy on this equilibrium using  perturbation analysis. The perturbed market is analyzed using the concept of closeness of two strategic games and the equilibria of close games using the notion of $\alpha-\text{approximation}$ and $\epsilon-\text{equilibrium}$. 
This analysis is used to quantify the effect of uncertainty of RERs and its possible mitigation using Demand Response in the form of
real time pricing, we quantify its effect on the market equilibrium using a parameter denoted as Curtailment Factor. We show that the equilibrium shift caused by the RER uncertainty is mitigated in the presence of RTP due to the Curtailment Factor.  Numerical results are included that validate the theoretical results, using an IEEE 30-bus network.


\end{document}